\documentclass[12pt, reqno, a4paper]{amsart}

\usepackage{amssymb}
\usepackage{hyperref}

\usepackage[english]{babel}

\theoremstyle{plain}

\newtheorem*{corollary}{Corollary}
\newtheorem{lemma}{Lemma}
\newtheorem{theorem}{Theorem}
\newtheorem*{conjecture}{Conjecture}

\theoremstyle{remark}
\newtheorem*{remark}{Remark}

\theoremstyle{definition}

\newtheorem{example}{Example}

\DeclareMathOperator{\Id}{Id}

\DeclareMathOperator{\ch}{char}

\DeclareMathOperator{\GL}{GL}

\DeclareMathOperator{\PGL}{PGL}

\DeclareMathOperator{\End}{End}
\DeclareMathOperator{\Aut}{Aut}
\DeclareMathOperator{\diag}{diag}

\DeclareMathOperator{\Alt}{Alt}

\DeclareMathOperator{\PIexp}{PIexp}

\begin{document}

\title{Algebras simple with respect to a Sweedler's algebra action}

\author{A.\,S.~Gordienko}
\address{Vrije Universiteit Brussel, Belgium}
\email{alexey.gordienko@vub.ac.be} 


\begin{abstract}
Algebras simple with respect to an action of Sweedler's algebra $H_4$
deliver the easiest example of $H$-module algebras that are $H$-simple
but not necessarily semisimple.
We describe finite dimensional $H_4$-simple algebras and prove the analog of Amitsur's conjecture
for codimensions of their polynomial $H_4$-identities.
In particular, we show that the Hopf PI-exponent of an $H_4$-simple algebra $A$ 
over an algebraically closed field of characteristic $0$ equals $\dim A$.
The groups of automorphisms preserving the structure of an $H_4$-module algebra are studied as well.
\end{abstract}


\thanks{Supported by Fonds Wetenschappelijk Onderzoek~--- Vlaanderen Pegasus Marie Curie post doctoral fellowship (Belgium) and RFBR grant 13-01-00234a (Russia).}

\maketitle

The notion of an $H$-(co)module algebra is a natural generalization of the notion of a graded algebra,
an algebra with an action of a group by automorphisms, and an algebra with an action of a Lie algebra by derivations.
 In particular, if $H_4$ is the $4$-dimensional Sweedler's algebra, an $H_4$-module
algebra is an algebra endowed both with an action of the cyclic group of order $2$ and with a skew-derivation satisfying certain conditions. 

The theory of gradings on matrix algebras and simple Lie algebras is a well developed area~\cite{ BahtKochMont, BahturinZaicevSeghal}. Quaternion $H_4$-extensions and related crossed products
were considered in~\cite{DoiTakeuchi}.
Here we classify finite dimensional $H_4$-simple algebras over an algebraically closed field $F$ of characteristic $\ch F \ne 2$ (Sections~\ref{SectionSweedlerSimpleMatrix}--\ref{SectionSweedlerSimpleNonSemisimple}). In addition,
we describe their automorphism groups as an $H_4$-module algebras (Section~\ref{SectionSweedlerSimpleAut}).

Amitsur's conjecture on asymptotic behaviour of codimensions of ordinary polynomial
identities was proved by A.~Giambruno and M.\,V.~Zaicev~\cite[Theorem~6.5.2]{ZaiGia} in 1999.

Suppose an algebra is endowed with a grading, an action of a group $G$ by automorphisms and anti-automorphisms, an action of a Lie algebra by derivations or a structure of an $H$-module algebra for some Hopf algebra $H$. Then
it is natural to consider, respectively, graded, $G$-, differential or $H$-identities~\cite{BahtGiaZai, BahtZaiGradedExp, BahturinLinchenko,  BereleHopf, Kharchenko}.

The analog of Amitsur's conjecture for polynomial $H$-identities was proved under wide conditions
by the author in~\cite{ASGordienko8, ASGordienko9}. However, in those results the $H$-invariance of the Jacobson radical was required. 
Until now the Sweedler's algebra~\cite[Section 7.4]{ASGordienko3} with the standard action of its dual was the only example where the analog of Amitsur's conjecture was proved for an $H$-simple
non-semisimple algebra.
 In this article we prove the analog of Amitsur's conjecture for all finite dimensional $H_4$-simple algebras not necessarily semisimple (Section~\ref{SectionSweedlerSimpleAmitsur}) assuming that the base field is algebraically closed and of characteristic $0$.

\section{Introduction}

An algebra $A$
over a field $F$
is an \textit{$H$-module algebra}
for some Hopf algebra $H$
if $A$ is endowed with a homomorphism $H \to \End_F(A)$ such that
$h(ab)=(h_{(1)}a)(h_{(2)}b)$
for all $h \in H$, $a,b \in A$. Here we use Sweedler's notation
$\Delta h = h_{(1)} \otimes h_{(2)}$ where $\Delta$ is the comultiplication
in $H$.
We refer the reader to~\cite{Danara, Montgomery, Sweedler}
   for an account
  of Hopf algebras and algebras with Hopf algebra actions.

Let $A$ be an $H$-module algebra for some Hopf algebra $H$ over a field $F$.
We say that $A$ is \textit{$H$-simple} if $A^2\ne 0$ and $A$ has no non-trivial
two-sided $H$-invariant ideals. 

By $H_4=\langle 1, c, v, cv \rangle_F$ we denote Sweedler's Hopf algebra over a field $F$ of characteristic $\ch F \ne 2$  where $$\Delta(c)=c\otimes c$$
and $$\Delta(v) = c\otimes v + v\otimes 1.$$

Note that $\lbrace 1, c \rbrace \cong \mathbb Z_2$ and every space $V$ with $\mathbb Z_2$-action has the following natural $\mathbb Z_2$-grading:
$V = V^{(0)}\oplus V^{(1)}$ where $cv^{(0)} = v^{(0)}$ for all $v^{(0)} \in V^{(0)}$
and $c v^{(1)} = -v^{(1)}$ for all $v^{(1)} \in V^{(1)}$.

\section{$H_4$-action on matrix algebras}\label{SectionSweedlerSimpleMatrix}

In this section we treat the case when a finite dimensional $H_4$-simple algebra $A$ is simple as an algebra. In the case of an algebraically closed field $F$ this implies that $A$ is isomorphic to a full matrix algebra over $F$.

\begin{theorem}\label{TheoremSweedlerSimpleMatrix}
Let $A$ be an $H_4$-module algebra over an algebraically closed field $F$, $\ch F \ne 2$, isomorphic as an algebra to $M_n(F)$ for some $n\in \mathbb N$.
Then either  
\begin{enumerate}
\item
$A$ is isomorphic as an algebra and an $H_4$-module
 to
$$M_{k,m}(F):=M_{k,m}^{(0)}(F)\oplus M_{k,m}^{(1)}(F)\text{ (direct sum of subspaces)}$$
where $c a = a$ for all $a \in M_{k,m}^{(0)}(F)$,
$c a = -a$ for all $a \in M_{k,m}^{(1)}(F)$,
$$M_{k,m}^{(0)}(F) = \left( \begin{array}{cc}
 M_k(F) & 0 \\
 0 & M_m(F)\\
 \end{array}\right),$$
 $$M_{k,m}^{(1)}(F) = \left( \begin{array}{cc}
 0 & M_{k\times m}(F) \\
 M_{m\times k}(F) & 0\\
 \end{array}\right),$$
and $va=(ca)Q-Qa$ where $Q=\left( \begin{array}{cc}
 0 & Q_1 \\
 Q_2 & 0\\
 \end{array}\right)$, $Q_1 Q_2 = \alpha E_k$, $Q_2 Q_1 = \alpha E_m$
  for some $k,m > 0$ (here $E_n$ is the identity matrix $n\times n$), $k+m = n$, $k\geqslant m$, $Q_1 \in M_{k\times m}(F)$,
  $Q_2 \in M_{m\times k}(F)$, $\alpha \in F$,
  \item or $ca=a$, $va=0$ for all $a\in A$.
  \end{enumerate}
\end{theorem}
\begin{proof}
Note that (see e.g, \cite[Section~3.5]{ZaiGia}) 
$A$ is isomorphic as a graded algebra to
$M_{k,m}(F)$ for some $k \geqslant m$, $k+m = n$.

 Moreover $v$ is acting on $A$ by a skew-derivation.
We claim\footnote{This result is a ``folklore'' one. I am grateful to V.\,K.~Kharchenko who informed me of a simple proof of it.} that this skew-derivation is \textit{inner},  i.e. there exists a matrix $Q \in A$ such that
$va = (ca)Q-Qa$ for all $a\in A$. Indeed, $ca = PaP^{-1}$
for all $a\in A$
where  $$P=\diag(\underbrace{1,1,\ldots, 1}_k, 
\underbrace{-1,-1,\ldots, -1}_m).$$
 Moreover $$P^{-1}(v(ab)) = P^{-1}((ca)(vb)+(va)b)
= P^{-1}((PaP^{-1})(vb)+(va)b)=a(P^{-1}(vb))+(P^{-1}(va))b$$
for all $a,b \in A$. Hence $P^{-1}(v(\cdot))$
is a derivation and $P^{-1}(va)=aQ_0-Q_0a$
for all $a\in A$ for some $Q_0 \in A$.
 Thus $$va = PaQ_0-PQ_0a= (PaP^{-1})PQ_0-PQ_0a= (ca)Q-Qa
 \text{ for all }a\in A$$ where $Q=PQ_0$, i.e. $v$ acts as an inner skew-derivation.
 
 Note that $vc=-cv$ implies $a(Q+cQ)=(Q+cQ)(ca)$ for all $a\in A$
 and $$Q+cQ= \beta\left( \begin{array}{cc}
 E_k &  0 \\
 0 & -E_m\\
 \end{array}\right)\text{ for some }\beta \in F.$$ Hence we may assume that $Q+cQ = 0$ and $Q=\left( \begin{array}{cc}
 0 & Q_1 \\
 Q_2 & 0
 \end{array}\right)$ since $$(ca)\left( \begin{array}{cc}
 E_k &  0 \\
 0 & -E_m\\
 \end{array}\right) - \left( \begin{array}{cc}
 E_k &  0 \\
 0 & -E_m\\
 \end{array}\right) a = 0 \text{ for all }a\in A.$$ 
 
 Recall that $v^2=0$. Therefore, for every $a\in A$, we
 have $$v^2 a = v((ca)Q-Qa)=(a(cQ)-(cQ)(ca))Q-Q((ca)Q-Qa)=Q^2 a-a Q^2$$
and $Q^2=\alpha E_n$ for some $\alpha \in F$.
Hence  $Q^2=\left( \begin{array}{cc}
 Q_1 Q_2 & 0\\
 0       & Q_2 Q_1
 \end{array}\right) = \left( \begin{array}{cc}
 \alpha E_k & 0\\
 0       & \alpha E_m
 \end{array}\right)$.
\end{proof}

\begin{remark}
Conversely, for every such $Q$, $k \geqslant m$, we have an 
$H_4$-module algebra isomorphic as an algebra to $M_n(F)$.
\end{remark}

\begin{remark}
If $k > m$, then $Q_1 Q_2$ is degenerate and we must have $Q^2=0$.
\end{remark}

\begin{theorem} \label{TheoremSweedlerSimpleMatrixIso}
Let $A_1$ and $A_2$ be $H_4$-module algebras over a field $F$, $\ch F \ne 2$,
$A_i=A_i^{(0)}\oplus A_i^{(1)}$ 
 (direct sum of subspaces)
where $c a = a$ for all $a \in A_i^{(0)}$,
$c a = -a$ for all $a \in A_i^{(1)}$,
$$A_i^{(0)} = \left( \begin{array}{cc}
 M_{k_i}(F) & 0 \\
 0 & M_{m_i}(F)\\
 \end{array}\right),$$
 $$A_i^{(1)} = \left( \begin{array}{cc}
 0 & M_{k_i\times m_i}(F) \\
 M_{m_i\times k_i}(F) & 0\\
 \end{array}\right),$$
and $va=(ca)Q_i-Q_ia$ for all $a\in A_i$, where $Q_i=\left( \begin{array}{cc}
 0 & Q_{i1} \\
 Q_{i2} & 0\\
 \end{array}\right)$,  $Q_{i1} Q_{i2} = \alpha_i E_{k_i}$, $Q_{i2} Q_{i1} = \alpha_i E_{m_i}$
  for some $k_i,m_i > 0$, $k_i\geqslant m_i$, $Q_{i1} \in M_{k_i\times m_i}(F)$,
  $Q_{i2} \in M_{m_i\times k_i}(F)$, $\alpha_i\in F$, $i=1,2$.
Then $A_1 \cong A_2$ as algebras and $H_4$-modules if and only if $k_1 = k_2$, $m_1 = m_2$
and 
there exist $W_1 \in \GL_{k_1}(F)$, $W_2 \in \GL_{m_1}(F)$
such that
\begin{enumerate}\item
 either $Q_{21}=W_1Q_{11}{W_2}^{-1}$ and $Q_{22}=W_2 Q_{12}{W_1}^{-1}$
or 
\item $k_1=m_1$,
$Q_{21}=W_1Q_{12}{W_2}^{-1}$, and $Q_{22}=W_2 Q_{11}{W_1}^{-1}$.
\end{enumerate}
\end{theorem}
\begin{proof} First we fix some isomorphisms $A_i \cong M_n(F)$, $i=1,2$, where
$n=k_i+m_i$.
Let $\varphi \colon A_1 \to A_2$ be an isomorphism of algebras and $H_4$-modules.
Then there exists $W\in\GL_n(F)$ such that $\varphi(a)=WaW^{-1}$
for all $a\in A_1$. Let $$P_i=\diag(\underbrace{1,1,\ldots, 1}_{k_i}, 
\underbrace{-1,-1,\ldots, -1}_{m_i}).$$
Then $\varphi(ca)=\varphi(P_1 a P_1^{-1})= c\varphi(a)= P_2 \varphi(a) P_2^{-1}$,
i.e. $WP_1 a (WP_1)^{-1}=P_2 W a (P_2 W)^{-1}$ for all $a\in A_1$.
Hence $W^{-1}P_2^{-1}WP_1$ commutes with any matrix.
Thus $W^{-1}P_2^{-1}WP_1 = \alpha E_n$ for some $\alpha \in F$.
Hence $P_1 = \alpha W^{-1} P_2 W$. Taking into account the eigenvalues of $P_1$ and $P_2$,
we get $\alpha = \pm 1$. In particular, $P_1=P_2$, $k_1 = k_2$ and $m_1 = m_2$.
 Note that $\varphi(va)=v\varphi(a)$ for all $a\in A_1$.
 Therefore $$W(P_1 a P_1^{-1} Q_1- Q_1 a)W^{-1}= P_1 WaW^{-1} P_1^{-1} Q_2- Q_2 WaW^{-1},$$
\begin{equation}\label{EqQW}a P_1^{-1} Q_1- P_1^{-1} Q_1 a= P_1^{-1}W^{-1}P_1 WaW^{-1} P_1^{-1} Q_2W- P_1^{-1}W^{-1}Q_2 Wa.\end{equation}

 If $\alpha = 1$, then $W P_1 = P_1 W$ and $W=\left(\begin{array}{cc}
 W_1 & 0 \\  0 & W_2 \end{array} \right)$ for some $W_1\in \GL_{k_1}(F)$, $W_2\in \GL_{m_1}(F)$.
 In this case,
 $$a P_1^{-1} Q_1- P_1^{-1} Q_1 a= aP_1^{-1}W^{-1}  Q_2W- P_1^{-1}W^{-1}Q_2 Wa,$$
 $$a (P_1^{-1} Q_1- P_1^{-1}W^{-1} Q_2W)= (P_1^{-1} Q_1- P_1^{-1}W^{-1}Q_2 W)a$$
 for all $a\in A_1$. Thus $(P_1^{-1} Q_1- P_1^{-1} W^{-1}Q_2W)$
 commutes with any $n\times n$ matrix and $P_1^{-1} Q_1- P_1^{-1}W^{-1} Q_2W = \beta E_n$
 for some $\beta \in F$, i.e. $Q_2 = W Q_1 W^{-1} - \beta P_1$.
 However, $Q_2$ and $W Q_1 W^{-1}$ have zeros on the main diagonal.
 Therefore $\beta = 0$ and  $Q_2 = W Q_1 W^{-1}$. In particular, 
$Q_{21}=W_1Q_{11}{W_2}^{-1}$ and $Q_{22}=W_2 Q_{12}{W_1}^{-1}$.

If $\alpha = -1$, then $k_1=m_1$, $P_1 W P_1^{-1} = -W$, and $W=\left(\begin{array}{cc}
 0 & W_1 \\  W_2 & 0 \end{array} \right)$ for some $W_1, W_2 \in \GL_{k_1}(F)$.
 Again~(\ref{EqQW}) implies $$a P_1^{-1} Q_1- P_1^{-1} Q_1 a= aP_1^{-1}W^{-1}  Q_2W - P_1^{-1}W^{-1}Q_2 Wa,$$ and $Q_2 = W Q_1 W^{-1} - \beta P_1$ for  $\beta \in F$.
 Note that $$W Q_1 W^{-1} = \left(\begin{array}{cc}
 0 & W_1 \\  W_2 & 0 \end{array} \right) \left(\begin{array}{cc}
 0 & Q_{11} \\  Q_{12} & 0 \end{array} \right)
 \left(\begin{array}{cc}
 0 & W_2^{-1} \\  W_1^{-1} & 0 \end{array} \right) = \left(\begin{array}{cc}
 0 & W_1 Q_{12} W_2^{-1} \\  W_2 Q_{11} W_1^{-1} & 0 \end{array} \right).$$
 Since the matrix $\beta P_1$ is diagonal, $\beta = 0$, $Q_{21}=W_1Q_{12}{W_2}^{-1}$, and $Q_{22}=W_2 Q_{11}{W_1}^{-1}$.

The direct assertion of the theorem is proved. Reversing the implications yields the converse.
 \end{proof}
 
 \begin{example}
 In the case of $2\times 2$ matrices we have the following variants:
 \begin{enumerate}
 \item $A = A^{(0)}= M_2(F)$, $A^{(1)}= 0$, $ca=a$, $va=0$ for all $a\in A$;
 \item $A = A^{(0)} \oplus A^{(1)}$ where $$A^{(0)}=\left\lbrace\left(\begin{array}{cc}
\alpha & 0 \\
0      & \beta 
 \end{array}\right) \mathbin{\biggl|} \alpha,\beta \in F\right\rbrace$$
 and $$A^{(1)}=\left\lbrace\left(\begin{array}{cc}
0 & \alpha \\
\beta      & 0 
 \end{array}\right) \mathbin{\biggl|} \alpha,\beta \in F\right\rbrace,$$
 $ca=(-1)^{i}a$, $va=0$ for $a\in A^{(i)}$;
  \item $A = A^{(0)} \oplus A^{(1)}$ where $$A^{(0)}=\left\lbrace\left(\begin{array}{cc}
\alpha & 0 \\
0      & \beta 
 \end{array}\right) \mathbin{\biggl|} \alpha,\beta \in F\right\rbrace$$
 and $$A^{(1)}=\left\lbrace\left(\begin{array}{cc}
0 & \alpha \\
\beta      & 0 
 \end{array}\right) \mathbin{\biggl|} \alpha,\beta \in F\right\rbrace,$$
 $ca=(-1)^{i}a$, $va = (ca)Q-Qa$ for $a\in A^{(i)}$ where $Q=\left(\begin{array}{cc}
0 & 1 \\
\gamma      & 0 
 \end{array}\right)$ and $\gamma \in F$ is a fixed number.
  \end{enumerate}
 \end{example}

\section{Semisimple $H_4$-simple algebras}

In this section we treat the case when an $H_4$-simple algebra $A$ is semisimple.

\begin{theorem}\label{TheoremSweedlerSimpleSemisimple}
Let $A$ be a semisimple $H_4$-simple algebra over an algebraically closed field $F$. Then either
\begin{enumerate}
\item
 $A$ is isomorphic to the full matrix algebra $M_k(F)$
for some $k \geqslant 1$ or
\item
 $A \cong M_k(F) \oplus M_k(F)$ (direct sum of ideals) for some $k \geqslant 1$
and there exists $P \in M_k(F)$ such that $P^2=\alpha E_k$ for some $\alpha \in F$, and
\begin{equation}\label{EqSSSweedlerSimple} c\, (a, b) = (b,a),\qquad v\,(a,b)=(Pa-bP,aP-Pb)\end{equation}
for all $a,b \in M_k(F)$.
\end{enumerate}
\end{theorem}
\begin{proof}
If $A$ is semisimple, then $A$ is the direct sum of $\mathbb Z_2$-simple subalgebras.
Let $B$ be one of such subalgebras.
Then $vb= v(1_B b)=(c1_B)(vb)+(v1_B)b \in B$
for all $b\in B$. Hence $B$ is an $H_4$-submodule, $A=B$,
and $A$ is a $\mathbb Z_2$-simple algebra.
Therefore,  either $A$ is isomorphic to the full matrix algebra $M_k(F)$
for some $k \geqslant 1$ or $A \cong M_k(F) \oplus M_k(F)$ (direct sum of ideals) for some $k \geqslant 1$
and $c\, (a, b) = (b,a)$ for all $a,b \in M_k(F)$.

In the first case, the theorem is proved. 
Consider the second case.
Let $\pi_1, \pi_2 \colon A \to M_k(F)$ be the natural projections on the first
and the second component, respectively. Consider $\rho_{ij} \in \End_F(M_k(F))$, $i,j 
\in \lbrace 1,2 \rbrace$,
defined by $\rho_{i1}(a):= \pi_i(v\,(a,0))$
and $\rho_{i2}(a):= \pi_i(v\,(0,a))$ for $a \in M_k(F)$.
Then \begin{equation*}\begin{split}\rho_{11}(ab)=\pi_1(v\,(ab,0))=\pi_1(v((a,0)(b,0)))=\\
\pi_1(c(a,0)\,v(b,0))+\pi_1((v(a,0))\,(b,0))= \\
\pi_1((0,a)\,v(b,0))+\pi_1(v(a,0))b=\rho_{11}(a)b
\end{split}\end{equation*}
for all $a,b \in M_k(F)$.
Analogously, $\rho_{12}(ab)=a\rho_{12}(b)$,
$\rho_{21}(ab)=a\rho_{21}(b)$, $\rho_{22}(ab)=\rho_{22}(a)b$
for all $a,b \in M_k(F)$.

Let $\rho_{11}(E_k)=P$, $\rho_{12}(E_k)=Q$, $\rho_{21}(E_k)=W$, $\rho_{22}(E_k)=T$ where
$P,Q,W,T \in M_k(F)$.
Then $$v(a,b)=(\pi_1(v\,(a,b)), \pi_2(v\,(a,b)))=
(\rho_{11}(a)+\rho_{12}(b), \rho_{21}(a)+\rho_{22}(b))
=(Pa+bQ, aW+Tb).$$
Now we notice that $$0=v((E_k,0)(0,E_k))=c(E_k,0)v(0,E_k)+(v(E_k,0))(0,E_k)=(0, T+W).$$
Thus $T=-W$.
Analogously, 
$$0=v((0,E_k)(E_k,0))=c(0,E_k)v(E_k,0)+(v(0,E_k))(E_k,0)=(P+Q, 0).$$
Hence $Q = - P$.
Moreover, $(W,P)=cv(E_k,0)=-vc(E_k,0)=-v(0,E_k)=(-Q,-T)$.
Therefore, $T=-P$, $W=P$,
and~(\ref{EqSSSweedlerSimple}) holds.

Note that $v^2(a,b)=([P^2,a], [P^2,b])$ for all $a,b\in M_k(F)$
where $[x,y]:=xy-yx$.
Since $v^2=0$, we obtain that $P^2$ is a scalar matrix
 and the theorem is proved.
\end{proof}
\begin{remark}
Conversely, for every matrix $P \in M_k(F)$
such that $P^2=\alpha E_k$ for some $\alpha \in F$,
we can define the structure of an $H_4$-simple algebra on 
$M_k(F) \oplus M_k(F)$ by~(\ref{EqSSSweedlerSimple}), which is even $\mathbb Z_2$-simple.
\end{remark}
\begin{theorem}\label{TheoremSweedlerSimpleSSIso}
Let $A_1 = M_k(F) \oplus M_k(F)$ be a semisimple $H_4$-simple algebra over a field $F$
defined by a matrix $P_1 \in M_k(F)$,
$P_1^2=\alpha_1 E_k$ for some $\alpha_1 \in F$, by~(\ref{EqSSSweedlerSimple}),
and let $A_2$ be another such algebra defined by a matrix $P_2 \in M_k(F)$.
Then $A_1 \cong A_2$ as algebras and $H_4$-modules
if and only if $P_2 = \pm Q P_1 Q^{-1}$
for some $Q \in \GL_k(F)$.
\end{theorem}
\begin{proof}
Note that in each of $A_1$ and $A_2$ there exist exactly two simple ideals 
isomorphic to $M_k(F)$. Moreover, each isomorphism of $M_k(F)$ is inner.
Therefore, if $\varphi \colon A_1 \to A_2$ is an isomorphism
of algebras and $H_4$-modules, then there exist 
matrices $Q, T \in \GL_k(F)$ such that
either $\varphi(a,b)=(Q a Q^{-1}, T b T^{-1})$
for all $a,b \in M_k(F)$ or 
$\varphi(a,b)=(Q b Q^{-1}, T a T^{-1})$
for all $a,b \in M_k(F)$.

Suppose $\varphi(a,b)=(Q a Q^{-1}, T b T^{-1})$
for all $a,b \in M_k(F)$.
Note that $c\varphi = \varphi c$ and $v\varphi = \varphi v$.
The first equality implies that $Q$ and $T$ are proportional, i.e without loss of generality we may assume that $Q=T$.
The second equality implies $(P_2, P_2) =v(E_k, 0)= v\varphi(E_k,0)=\varphi(v(E_k,0))=\varphi(P_1,P_1)=
(Q P_1 Q^{-1}, Q P_1 Q^{-1})$. Therefore, $P_2 = Q P_1 Q^{-1}$.

Suppose $\varphi(a,b)=(Q b Q^{-1}, T a T^{-1})$
for all $a,b \in M_k(F)$. Again, we may assume that $Q=T$.
Analogously, $(-P_2, -P_2)=v(0,E_k)=v\varphi(E_k,0)=\varphi (v(E_k,0))=\varphi(P_1, P_1)=
(Q P_1 Q^{-1}, Q P_1 Q^{-1})$ implies $P_2 = -Q P_1 Q^{-1}$.

The converse assertion is proved explicitly.
\end{proof}

\begin{remark}
Theorems~\ref{TheoremSweedlerSimpleSemisimple} and~\ref{TheoremSweedlerSimpleSSIso}
imply that every semisimple $H_4$-simple algebra $A$ over an algebraically closed field $F$,
$\ch F \ne 2$, that is not simple as an ordinary algebra,
is isomorphic to
$M_n(F) \oplus M_n(F)$ (direct sum of ideals) for some $n \geqslant 1$
where
$$ c\, (a, b) = (b,a),\qquad v\,(a,b)=(Pa-bP,aP-Pb)$$
for all $a,b \in M_n(F)$
and
\begin{enumerate}
\item
either $P=(\underbrace{\alpha,\alpha,\ldots, \alpha}_m, 
\underbrace{-\alpha,-\alpha,\ldots, -\alpha}_k)$ for some $\alpha \in F$
and $m \geqslant k$, $m+k=n$,
\item or $P$ is a block diagonal matrix with several blocks $\left(\begin{smallmatrix} 
 0 & 1 \\
 0 & 0\\
 \end{smallmatrix}\right)$ on the main diagonal (the rest cells are filled with zero)
 \end{enumerate}
and these algebras are not isomorphic for different $P$.
\end{remark}

\section{Non-semisimple algebras}\label{SectionSweedlerSimpleNonSemisimple}

In this section we consider the last possibility, i.e. when $H_4$-simple algebra has a non-trivial
radical.

\begin{theorem}\label{TheoremSweedlerNonSemiSimple}
Let $A$ be a finite dimensional associative $H_4$-module algebra over a field $F$
of characteristic $\ch F\ne 2$.
Suppose $A$ is $H_4$-simple and not semisimple.
Let $G=\lbrace 1, c \rbrace \cong \mathbb Z_2$.
Then there exists a $G$-invariant ideal $J \subset A$
such that $A = vJ \oplus J$ (direct sum of $G$-invariant subspaces), $J^2 = 0$,
$vJ$ is a $G$-simple algebra. Moreover, there exists a linear isomorphism $\varphi \colon vJ \to J$
such that $v\varphi(b)=b$, $a \varphi(b) = \varphi((ca)b)$,
$\varphi(a)b = \varphi(ab)$ for all $a,b\in vJ$,
 and $\varphi((vJ)^{(0)})=J^{(1)}$, $\varphi((vJ)^{(1)})=J^{(0)}$.
\end{theorem}
\begin{proof}
Let $J$ be the Jacobson radical of $A$. Since $J$ is invariant under all automorphisms,
$J$ is a $G$-invariant subspace. Therefore, $vA \ne 0$ since otherwise $J$ would be an $H_4$-invariant ideal.

 Let $p\in \mathbb N$ be a number such that $J^p = 0$, $J^{p-1}\ne 0$. 
Note that $J^{p-1} + vJ^{p-1}$ is an $H_4$-invariant two-sided ideal of $A$.
Indeed, $$a(vj)=(c(ca))(vj)=v((ca)j)-(v(ca)j) \in J^{p-1} + vJ^{p-1}$$
and $$(vj)a=v(ja)-(cj)(va) \in J^{p-1} + vJ^{p-1}$$
for all $a \in A$ and $j\in J^{p-1}$.
Hence $A=J^{p-1} + vJ^{p-1}$. 

Note that $(J^{p-1})^2 = 0$.
Moreover, $vJ^{p-1}$ is a subalgebra since
 $(vj_1)(vj_2)= v(j_1 (vj_2)) - (cj_1)(v^2 j_2) = v(j_1 (vj_2)) \in vJ^{p-1}$
 for all $j_1,j_2 \in J^{p-1}$.

Let $N := \lbrace a \in A \mid va = 0 \rbrace \cap J^{p-1}$. Clearly, $N$ is $H_4$-invariant.
We claim that $N$ is a two-sided ideal. Indeed, let $j\in J^{p-1}$ and $b\in N$.
Then $v((vj)b)=(cvj)(vb)+(v^2j)b=0+0=0$ and 
$v(b(vj))=(cb)(v^2j)+(vb)(vj)=0+0=0$. Therefore, $(vJ^{p-1})N=N(vJ^{p-1})=0$.
Since $J^{p-1}N=NJ^{p-1}=(J^{p-1})^2=0$, the space $N$ is an $H_4$-invariant ideal of $A$.

Recall that $vA\ne 0$. Thus $N\ne A$ and $N=0$.
Hence $A=J^{p-1} \oplus vJ^{p-1}$ (direct sum of $G$-invariant subspaces)
and the multiplication by $v$ is a linear isomorphism from $J^{p-1}$
onto $vJ^{p-1}$. Denote by $\varphi \colon v J^{p-1} \to J^{p-1}$ the inverse map.
Then $$a\varphi(b)=\varphi(v(a\varphi(b)))=\varphi((ca)(v\varphi(b)))+\varphi((va)\varphi(b))
=\varphi((ca)b)$$ and
$$\varphi(a)b=\varphi(v(\varphi(a)b))=\varphi((c\varphi(a))(vb))+\varphi((v\varphi(a))b)
=\varphi(ab)$$ for all $a,b\in vJ^{p-1}$. Therefore, $A^2\ne 0$ implies  $(vJ^{p-1})^2\ne 0$.

We claim that the subalgebra $vJ^{p-1}$ is $G$-simple. Suppose there exists a $G$-invariant
ideal $I$ of $vJ^{p-1}$. We claim that $I+\varphi(I)$ is an $H_4$-invariant ideal in $A$.
Indeed, $$c(I+\varphi(I))=cI-\varphi(cI)=I+\varphi(I).$$
 Moreover,
$$a(I+\varphi(I))=aI+\varphi((ca)I)\subseteq I+\varphi(I),$$
$$(I+\varphi(I))a=Ia+\varphi(Ia)\subseteq I+\varphi(I),$$
$$\varphi(a)(I+\varphi(I))=\varphi(aI)\subseteq I+\varphi(I),$$
and $$(I+\varphi(I))\varphi(a)=\varphi((cI)a)\subseteq I+\varphi(I)$$
for all $a \in vJ^{p-1}$. Hence $I+\varphi(I)$ is an $H_4$-invariant ideal in $A$,
$I = vJ^{p-1}$, and $vJ^{p-1}$ is a $G$-simple algebra.

In particular, $vJ^{p-1}$ is semisimple. Hence $A/J^{p-1} \cong vJ^{p-1}$ is a semisimple algebra
and $J=J^{p-1}$ is the Jacobson radical of $A$. All the assertions of the theorem have been proved.
\end{proof}

\begin{theorem}
Let $B$ be a $\mathbb Z_2$-simple algebra over a field $F$, $\ch F\ne 2$, where $\mathbb Z_2 = \lbrace 1, c\rbrace$
and $c$ acts as an automorphism and
let $\varphi \colon B \to J$ be an isomorphism from $B$ to a vector space $J$.
We define $\mathbb Z_2$-action on $J$ by the formulas $J^{(0)}=\varphi(B^{(1)})$
and $J^{(1)}=\varphi(B^{(0)})$. In addition, we define multiplication
in  $A = B \oplus J$ (direct sum of $\mathbb Z_2$-invariant subspaces)
by $a \varphi(b) = \varphi((ca)b)$,
$\varphi(a)b = \varphi(ab)$ for all $a,b\in vJ$, $J^2=0$.
Then $A$ is an $H_4$-module algebra for the Sweedler algebra $H_4=\langle 1, c, v, cv \rangle_F$
where $v(a+\varphi(b))=b$ for all $a,b\in B$.
Moreover, $A$ is $H_4$-simple.
\end{theorem}
\begin{proof}
Trivial checks show that $A$ is indeed an $H_4$-module algebra.
Now we prove that $A$ is $H_4$-simple.

Let $I$ be an $H_4$-invariant ideal of $A$. Then $vI \subseteq B \cap I$.
Note that $\mathbb Z_2$-simplicity of $B$ implies that either $B \cap I = 0$
or $B \subseteq I$. However if $B \cap I = 0$, then $vI = 0$, $I\subseteq B$, and $I=0$.
Thus we may assume that $B \subseteq I$. Hence $a+\varphi(b)=a+\varphi(1)b\in I$
for all $a,b\in B$. Therefore, $I=A$ and $A$ is $H_4$-simple.
\end{proof}
\begin{remark}
Two algebras $A$ are isomorphic if and only if their maximal $\mathbb Z_2$-simple
subalgebras $B$ are isomorphic as $\mathbb Z_2$-algebras.
\end{remark}

\section{Automorphisms of $H_4$-simple algebras}\label{SectionSweedlerSimpleAut}

In this section we study the automorphisms of finite dimensional $H_4$-simple algebras
over a field $F$, $\ch F \ne 2$. In the first five cases below the automorphism groups can be determined explicitly. For the last case, we present a method of description of the automorphism group.

\begin{enumerate}
\item \label{ItemSweedlerSimpleMkl} if $A=M_{k,\ell}(F)$ for some $k > \ell$,  $c a = a$ for all $a \in M_{k,\ell}^{(0)}(F)$,
$c a = -a$ for all $a \in M_{k,\ell}^{(1)}(F)$, $va=0$ for all $a\in A$, then
$\Aut_{H_4}(A) \cong (\GL_k(F)\times \GL_\ell(F))/F^{\times}E_{k+\ell}$
where  $F^{\times}E_n$ is the group of scalar $n\times n$ matrices
(this follows from the beginning of the proof of Theorem~\ref{TheoremSweedlerSimpleMatrixIso});

\item \label{ItemSweedlerSimpleMkk} if $A=M_{k,k}(F)$ for some $k\in \mathbb N$,  $c a = a$ for all $a \in M_{k,k}^{(0)}(F)$,
$c a = -a$ for all $a \in M_{k,k}^{(1)}(F)$, $va=0$ for all $a\in A$, then
$$\Aut_{H_4}(A) \cong \left(\bigl(\GL_k(F)\times \GL_k(F)\bigr) \leftthreetimes \left\langle\left(\begin{smallmatrix}0 & E_k \\ E_k & 0 \end{smallmatrix}\right)\right\rangle_2 \right)/F^{\times}E_{2k}$$ 
(this follows from the proof of Theorem~\ref{TheoremSweedlerSimpleMatrixIso} too);

\item \label{ItemSweedlerSimpleDoubleMk} Let $A \cong M_k(F) \oplus M_k(F)$  (direct sum of ideals) be an $H_4$-simple algebra  for some $k \geqslant 1$ and
$$ c\, (a, b) = (b,a),\qquad v\,(a,b)=(0,0)$$
for all $a,b \in M_k(F)$.
Then $\Aut_{H_4}(A) \cong \PGL_k(F) \times \mathbb Z_2$ 
(this follows from the proof of Theorem~\ref{TheoremSweedlerSimpleSSIso});

\item if $A$ is a finite dimensional non-semisimple $H_4$-simple algebra,
then every $H_4$-automorphism $\psi$ stabilizes the radical $J$ of $A$ and therefore
the $\mathbb Z_2$-simple subalgebra $vJ$ of $A$. Furthermore,
 $\psi(va)=v\psi(a)$ implies that $\psi$ is completely determined by
its restriction on $vJ$. Therefore, $\Aut_{H_4}(A)=\Aut_{\mathbb Z_2}(vJ)$.
But $\Aut_{\mathbb Z_2}(vJ)$ is described in (\ref{ItemSweedlerSimpleMkl})--(\ref{ItemSweedlerSimpleDoubleMk}).

\item  \label{ItemSweedlerSimpleSSP}
Let $A \cong M_n(F) \oplus M_n(F)$ (direct sum of ideals), $n \geqslant 1$, be a semisimple $H_4$-simple algebra
where 
$$ c\, (a, b) = (b,a),\qquad v\,(a,b)=(Pa-bP,aP-Pb)$$
for all $a,b \in M_n(F)$ where $P \in M_n(F)$ is a fixed matrix such that $P^2=\alpha E_n$ for some $\alpha \in F$.  Then the proof of Theorem~\ref{TheoremSweedlerSimpleSSIso}
  implies that $\Aut_{H_4}(A)$ is isomorphic to the subgroup of $\PGL_{n}(F)$ 
  that consists of the images of all $n\times n$ matrices that commute or anti-commute with $P$.
  Note that $M_{m,k}(F)^{(1)}$ contains invertible matrices only for $m=k$.
  Therefore, 
\begin{enumerate}
\item  
  if $P=(\underbrace{\alpha,\alpha,\ldots, \alpha}_m, 
\underbrace{-\alpha,-\alpha,\ldots, -\alpha}_k)$ for some $\alpha \in F^\times$
and $m \geqslant k$, $m+k=n$, then $\Aut_{H_4}(A)$ is isomorphic
to the subgroup of $\PGL_{n}(F)$ that consists of the images of all invertible matrices
 from $M_{m,k}(F)^{(0)}$ for $m > k$ and from $M_{k,k}(F)^{(0)} \cup M_{k,k}(F)^{(1)}$ for $m=k$.
\item If $$P=\left(
\begin{array}{ccccccc}
\begin{array}{|cc|}
\hline
0 & 1 \\
0 & 0 \\
\hline
\end{array} & \begin{array}{cc}
0 & 0 \\
0 & 0 \\
\end{array} & \ldots & \begin{array}{cc}
0 & 0 \\
0 & 0 \\
\end{array} & \begin{array}{c}
0  \\
0  
\end{array}  & \ldots & \begin{array}{c}
0  \\
0  
\end{array}\\
       \begin{array}{cc}
0 & 0 \\
0 & 0 \\
\end{array}    & \begin{array}{|cc|}
                  \hline
                   0 & 1 \\
                    0 & 0 \\
                \hline
                \end{array} &  \ldots & \begin{array}{cc}
0 & 0 \\
0 & 0 \\
\end{array} &  \begin{array}{c}
0  \\
0  
\end{array} & \ldots & \begin{array}{c}
0  \\
0  
\end{array}\\
\multicolumn{7}{c}{\dotfill}\\
\begin{array}{cc}
0 & 0 \\
0 & 0 \\
\end{array} & \begin{array}{cc}
0 & 0 \\
0 & 0 \\
\end{array} & \ldots & \begin{array}{|cc|}
\hline
0 & 1 \\
0 & 0 \\
\hline
\end{array} & \begin{array}{c}
0  \\
0  
\end{array}  & \ldots & \begin{array}{c}
0  \\
0  
\end{array}\\
\begin{array}{cc}
0  & 0  
\end{array} & \begin{array}{cc}
0  & 0  
\end{array} & \ldots & \begin{array}{cc}
0  & 0  
\end{array} & 0 & \ldots & 0 \\ 
\multicolumn{7}{c}{\dotfill}\\
\begin{array}{cc}
0  & 0  
\end{array} & \begin{array}{cc}
0  & 0  
\end{array} & \ldots & \begin{array}{cc}
0  & 0  
\end{array} & 0 & \ldots & 0 \\ 
\end{array}
\right)$$
where the number of cells $\left(\begin{smallmatrix} 
 0 & 1 \\
 0 & 0\\
 \end{smallmatrix}\right)$ equals $k\in\mathbb N$,
 then~\cite[Section 69]{MalcevAIFoundations} implies
 that $\Aut_{H_4}(A)$ is isomorphic
to the subgroup of $\PGL_{n}(F)$ that consists of
all invertible matrices
$$\left(
\begin{array}{cc|cc|c|cc|ccc}
 \alpha_{11} & \beta_{11} & \alpha_{12} & \beta_{12} &   \ldots & \alpha_{1k} & \beta_{1k}
 & \gamma_{1,k+1} & \ldots &  \gamma_{1,{n-2k}}\\
 0           &  \alpha_{11} &       0   & \alpha_{12} &  \ldots &  0 & \alpha_{1k}
 & 0 & \ldots &  0 \\
\hline
 \alpha_{21} & \beta_{21} & \alpha_{22} & \beta_{22} &   \ldots & \alpha_{2k} & \beta_{2k}
 & \gamma_{2,k+1} & \ldots &  \gamma_{2,{n-2k}}\\
 0           &  \alpha_{21} &       0   & \alpha_{22} &  \ldots &  0 & \alpha_{2k}
 & 0 & \ldots &  0 \\
\hline
\multicolumn{2}{c|}{\dotfill} & \multicolumn{2}{c|}{\dotfill} & \ldots & \multicolumn{2}{c|}{\dotfill}
& \multicolumn{3}{c}{\dotfill} \\
\hline
 \alpha_{k1} & \beta_{k1} & \alpha_{k2} & \beta_{k2} &   \ldots & \alpha_{kk} & \beta_{kk}
 & \gamma_{k,k+1} & \ldots &  \gamma_{k,{n-2k}}\\
 0           &  \alpha_{k1} &       0   & \alpha_{k2} &  \ldots &  0 & \alpha_{kk}
 & 0 & \ldots &  0 \\
\hline
0 & \gamma_{k+1, 1} & 0 & \gamma_{k+1, 2} & \ldots & 0 &  \gamma_{k+1, k}
&  \gamma_{k+1, k+1} & \ldots & \gamma_{k+1, n-2k} \\
\multicolumn{2}{c|}{\dotfill} & \multicolumn{2}{c|}{\dotfill} & \ldots & \multicolumn{2}{c|}{\dotfill}
& \multicolumn{3}{c}{\dotfill} \\
0 & \gamma_{n-2k, 1} & 0 & \gamma_{n-2k, 2} & \ldots & 0 &  \gamma_{n-2k, k}
&  \gamma_{n-2k, k+1} & \ldots & \gamma_{n-2k, n-2k} \\
\end{array}
\right)$$
and 
$$\left(
\begin{array}{cc|cc|c|cc|ccc}
 \alpha_{11} & \beta_{11} & \alpha_{12} & \beta_{12} &   \ldots & \alpha_{1k} & \beta_{1k}
 & \gamma_{1,k+1} & \ldots &  \gamma_{1,{n-2k}}\\
 0           &  -\alpha_{11} &       0   & -\alpha_{12} &  \ldots &  0 & -\alpha_{1k}
 & 0 & \ldots &  0 \\
\hline
 \alpha_{21} & \beta_{21} & \alpha_{22} & \beta_{22} &   \ldots & \alpha_{2k} & \beta_{2k}
 & \gamma_{2,k+1} & \ldots &  \gamma_{2,{n-2k}}\\
 0           &  -\alpha_{21} &       0   & -\alpha_{22} &  \ldots &  0 & -\alpha_{2k}
 & 0 & \ldots &  0 \\
\hline
\multicolumn{2}{c|}{\dotfill} & \multicolumn{2}{c|}{\dotfill} & \ldots & \multicolumn{2}{c|}{\dotfill}
& \multicolumn{3}{c}{\dotfill} \\
\hline
 \alpha_{k1} & \beta_{k1} & \alpha_{k2} & \beta_{k2} &   \ldots & \alpha_{kk} & \beta_{kk}
 & \gamma_{k,k+1} & \ldots &  \gamma_{k,{n-2k}}\\
 0           &  -\alpha_{k1} &       0   & -\alpha_{k2} &  \ldots &  0 & -\alpha_{kk}
 & 0 & \ldots &  0 \\
\hline
0 & \gamma_{k+1, 1} & 0 & \gamma_{k+1, 2} & \ldots & 0 &  \gamma_{k+1, k}
&  \gamma_{k+1, k+1} & \ldots & \gamma_{k+1, n-2k} \\
\multicolumn{2}{c|}{\dotfill} & \multicolumn{2}{c|}{\dotfill} & \ldots & \multicolumn{2}{c|}{\dotfill}
& \multicolumn{3}{c}{\dotfill} \\
0 & \gamma_{n-2k, 1} & 0 & \gamma_{n-2k, 2} & \ldots & 0 &  \gamma_{n-2k, k}
&  \gamma_{n-2k, k+1} & \ldots & \gamma_{n-2k, n-2k} \\
\end{array}
\right).$$
\end{enumerate}

\item \label{ItemSweedlerSimpleMatrixQ} Suppose $A=M_{k,m}(F)$ be an $H_4$-module algebra for some $k,m\in \mathbb N$, $k\geqslant m$,
where $c a = a$ for all $a \in M_{k,m}^{(0)}(F)$,
$c a = -a$ for all $a \in M_{k,m}^{(1)}(F)$,
and $va=(ca)Q-Qa$ where $Q=\left( \begin{array}{cc}
 0 & Q_1 \\
 Q_2 & 0\\
 \end{array}\right)$, $Q_1 Q_2 = \alpha E_k$, $Q_2 Q_1 = \alpha E_m$, $Q_1 \in M_{k\times m}(F)$,
  $Q_2 \in M_{m\times k}(F)$, $\alpha \in F$.
  Then the proof of Theorem~\ref{TheoremSweedlerSimpleMatrixIso}
  implies that 
\begin{enumerate}
\item  if $k\ne m$, then
  $\Aut_{H_4}(A)$ is isomorphic to the subgroup of $\PGL_{k+m}(F)$
  that consists of the images of all matrices $W=\left(\begin{array}{cc}
  W_1 & 0 \\  0 & W_2 \end{array} \right)$, $W_1 \in \GL_k(F)$, $W_2\in \GL_m(F)$, that commute with $Q$,
  \item  
  and if $k=m$, then
  $\Aut_{H_4}(A)$ is isomorphic to the subgroup of $\PGL_{2k}(F)$
  that consists of the images of all matrices $W=\left(\begin{array}{cc}
  W_1 & 0 \\  0 & W_2 \end{array} \right)$ and $W=\left(\begin{array}{cc}
 0 & W_1 \\  W_2 & 0 \end{array} \right)$, $W_1, W_2\in \GL_k(F)$, that commute with $Q$.
\end{enumerate}
  Again, for any particular $Q$, the matrices $W$ can be determined using the Jordan normal form of $Q$
 (see e.g.~\cite[Section 69]{MalcevAIFoundations}).
\end{enumerate}

In the examples below we treat some particular cases of~(\ref{ItemSweedlerSimpleSSP})
and~(\ref{ItemSweedlerSimpleMatrixQ}).

\begin{example}
Let $A=M_{1,1}(F)$ where $c a = a$ for all $a \in M_{1,1}^{(0)}(F)=\left\lbrace\left(
\begin{smallmatrix} 
\alpha & 0 \\
0 & \beta
\end{smallmatrix}
 \right)\right\rbrace$,
$c a = -a$ for all $a \in M_{1,1}^{(1)}(F)=\left\lbrace\left(
\begin{smallmatrix} 
0 & \alpha \\
\beta & 0
\end{smallmatrix}
 \right)\right\rbrace$,
and $va=(ca)Q-Qa$ where $Q=\left(\begin{smallmatrix} 
 0 & 1 \\
 1 & 0\\
 \end{smallmatrix}\right)$.
  We notice that all matrices from $M_{1,1}^{(0)}(F)$ and $M_{1,1}^{(1)}(F)$
  that commute with $Q$, are of the form $\left(\begin{smallmatrix} 
 \alpha & 0 \\
 0 & \alpha\\
 \end{smallmatrix}\right)$ and $\left(\begin{smallmatrix} 
 0 & \alpha \\
 \alpha & 0\\
 \end{smallmatrix}\right)$, respectively, where $\alpha \in F^\times$.
 After factoring by the subgroup of scalar matrices, we get $\Aut_{H_4}(A)\cong \mathbb Z_2$.
\end{example}

\begin{example}
Let $A \cong M_2(F) \oplus M_2(F)$ (direct sum of ideals)
 be an $H_4$-simple algebra over a field $F$,
 $$ c\, (a, b) = (b,a),\qquad v\,(a,b)=(Pa-bP,aP-Pb)$$
for all $a,b \in M_k(F)$ where $P=\left(\begin{smallmatrix} 
 0 & 1 \\
 0 & 0 \\
 \end{smallmatrix}\right)$, $\lambda \in F$.
 All matrices commuting with $P$ are of the form $\left(\begin{smallmatrix} 
 \mu & \gamma \\
 0 & \mu \\
 \end{smallmatrix}\right)$, $\gamma \in F$, $\mu\in F^\times$.
  All matrices anti-commuting with $P$ are of the form $\left(\begin{smallmatrix} 
 \mu & \gamma \\
 0 & -\mu \\
 \end{smallmatrix}\right)$, $\gamma \in F$, $\mu\in F^\times$.
 After factoring by the subgroup of scalar matrices, we get $$\Aut_{H_4}(A)\cong 
 \left\lbrace \left(\begin{smallmatrix} 
 1 & \mu \\
 0 & 1 
 \end{smallmatrix} \right) \mathbin{\bigr|} \mu \in F  \right\rbrace
 \leftthreetimes \left\langle \left(\begin{smallmatrix} 
 1 & 0 \\
 0 & -1 
 \end{smallmatrix} \right) \right\rangle_2 \cong
(F, +) \leftthreetimes \mathbb Z_2$$
where the result of the conjugation of an element $\alpha \in (F, +)$ by the generator of $\mathbb Z_2$ is  $(-\alpha)$.
\end{example}

\begin{example}
Let $A \cong M_2(F) \oplus M_2(F)$ (direct sum of ideals)
 be an $H_4$-simple algebra over a field $F$,
 $$ c\, (a, b) = (b,a),\qquad v\,(a,b)=(Pa-bP,aP-Pb)$$
for all $a,b \in M_k(F)$ where $P=\left(\begin{smallmatrix} 
 1 & 0 \\
 0 & -1 \\
 \end{smallmatrix}\right)$. All matrices commuting with $P$ are diagonal,
 all matrices anti-commuting with $P$ are anti-diagonal.
 After factoring by the subgroup of scalar matrices, we get $$\Aut_{H_4}(A) \cong \left\lbrace \left(\begin{smallmatrix} 
 1 & 0 \\
 0 & \alpha 
 \end{smallmatrix} \right) \mathbin{\bigr|} \alpha \in F^\times  \right\rbrace
 \leftthreetimes \left\langle \left(\begin{smallmatrix} 
 0 & 1 \\
 1 & 0 
 \end{smallmatrix} \right) \right\rangle_2\cong
 F^\times \leftthreetimes \mathbb Z_2$$ where the conjugation of an element of $F^\times$
 by the generator of $\mathbb Z_2$ is the inversion.
\end{example}

\section{Growth of polynomial $H_4$-identities}\label{SectionSweedlerSimpleAmitsur}

Here we apply the results of Section~\ref{SectionSweedlerSimpleNonSemisimple}
to polynomial $H_4$-identities.

First we introduce the notion of the free associative $H$-module algebra. Here we follow~\cite{BahturinLinchenko}.
Let $F \langle X \rangle$ be the free associative algebra without $1$
   on the set $X := \lbrace x_1, x_2, x_3, \ldots \rbrace$.
  Then $F \langle X \rangle = \bigoplus_{n=1}^\infty F \langle X \rangle^{(n)}$
  where $F \langle X \rangle^{(n)}$ is the linear span of all monomials of total degree $n$.
   Let $H$ be a Hopf algebra over $F$. Consider the algebra $$F \langle X | H\rangle
   :=  \bigoplus_{n=1}^\infty H^{{}\otimes n} \otimes F \langle X \rangle^{(n)}$$
   with the multiplication $(u_1 \otimes w_1)(u_2 \otimes w_2):=(u_1 \otimes u_2) \otimes w_1w_2$
   for all $u_1 \in  H^{{}\otimes j}$, $u_2 \in  H^{{}\otimes k}$,
   $w_1 \in F \langle X \rangle^{(j)}$, $w_2 \in F \langle X \rangle^{(k)}$.
We use the notation $$x^{h_1}_{i_1}
x^{h_2}_{i_2}\ldots x^{h_n}_{i_n} := (h_1 \otimes h_2 \otimes \ldots \otimes h_n) \otimes x_{i_1}
x_{i_2}\ldots x_{i_n}.$$ Here $h_1 \otimes h_2 \otimes \ldots \otimes h_n \in H^{{}\otimes n}$,
$x_{i_1} x_{i_2}\ldots x_{i_n} \in F \langle X \rangle^{(n)}$. 

Note that if $(\gamma_\beta)_{\beta \in \Lambda}$ is a basis in $H$, 
then $F\langle X | H \rangle$ is isomorphic to the free associative algebra over $F$ with free formal  generators $x_i^{\gamma_\beta}$, $\beta \in \Lambda$, $i \in \mathbb N$.
 We refer to the elements
 of $F\langle X | H \rangle$ as \textit{associative $H$-polynomials}.
 
In addition, $F \langle X | H \rangle$ becomes an $H$-module algebra with the $H$-action
defined by
$h(x^{h_1}_{i_1}
x^{h_2}_{i_2}\ldots x^{h_n}_{i_n})=x^{h_{(1)}{h_1}}_{i_1} x^{h_{(2)}{h_2}}_{i_2}\ldots
x^{h_{(n)}{h_n}}_{i_n}$ for $h\in H$.

Let $A$ be an associative $H$-module algebra.
Any map $\psi \colon X \to A$ has a unique homomorphic extension $\bar\psi
\colon F \langle X | H \rangle \to A$ such that $\bar\psi(h w)=h\psi(w)$
for all $w \in F \langle X | H \rangle$ and $h \in H$.
 An $H$-polynomial
 $f \in F\langle X | H \rangle$
 is an \textit{$H$-identity} of $A$ if $\varphi(f)=0$
for all homomorphisms $\varphi \colon F \langle X | H \rangle \to A$
of algebras and $H$-modules.
 In other words, $f(x_1, x_2, \ldots, x_n)$
 is an $H$-identity of $A$
if and only if $f(a_1, a_2, \ldots, a_n)=0$ for any $a_i \in A$.
 In this case we write $f \equiv 0$.
The set $\Id^{H}(A)$ of all $H$-identities
of $A$ is an $H$-invariant ideal of $F\langle X | H \rangle$.

We denote by $P^H_n$ the space of all multilinear $H$-polynomials
in $x_1, \ldots, x_n$, $n\in\mathbb N$, i.e.
$$P^{H}_n = \langle x^{h_1}_{\sigma(1)}
x^{h_2}_{\sigma(2)}\ldots x^{h_n}_{\sigma(n)}
\mid h_i \in H, \sigma\in S_n \rangle_F \subset F \langle X | H \rangle.$$
Then the number $c^H_n(A):=\dim\left(\frac{P^H_n}{P^H_n \cap \Id^H(A)}\right)$
is called the $n$th \textit{codimension of polynomial $H$-identities}
or the $n$th \textit{$H$-codimension} of $A$.

The analog of Amitsur's conjecture for $H$-codimensions can be formulated
as follows.

\begin{conjecture} There exists
 $\PIexp^H(A):=\lim\limits_{n\to\infty}
 \sqrt[n]{c^H_n(A)} \in \mathbb Z_+$.
\end{conjecture}

In the theorem below we consider the case $H=H_4$.

\begin{theorem}\label{TheoremSweedlerSimpleAmitsur} Let $A$ be a finite dimensional $H_4$-simple
algebra over an algebraically closed field $F$ of characteristic $0$.
Then there exist constants $C > 0$, $r\in \mathbb R$ such that $$C n^{r} (\dim A)^n \leqslant c^{H_4}_n(A) \leqslant (\dim A)^{n+1}\text{ for all }n \in \mathbb N.$$
\end{theorem}
\begin{corollary}
The analog of Amitsur's conjecture holds
 for such codimensions. In particular, $\PIexp^H(A)=\dim A$.
\end{corollary}

In order to prove Theorem~\ref{TheoremSweedlerSimpleAmitsur}, we need one lemma.

Let $k\ell \leqslant n$ where $k,n \in \mathbb N$ are some numbers.
 Denote by $Q^H_{\ell,k,n} \subseteq P^H_n$
the subspace spanned by all $H$-polynomials that are alternating in
$k$ disjoint subsets of variables $\{x^i_1, \ldots, x^i_\ell \}
\subseteq \lbrace x_1, x_2, \ldots, x_n\rbrace$, $1 \leqslant i \leqslant k$.

\begin{lemma}\label{LemmaSweedlerNSSAltHPolynomial}
Let $A$ be an $H_4$-simple non-semisimple associative algebra over an 
algebraically closed field $F$
of characteristic $0$, $\dim A=2\ell$.
Then there exists a number $n_0 \in \mathbb N$ such that for every $n\geqslant n_0$
there exist disjoint subsets $X_1$, \ldots, $X_k \subseteq \lbrace x_1, \ldots, x_n
\rbrace$, $k = \left[\frac{n-n_0}{2\ell}\right]$,
$|X_1| = \ldots = |X_{k}|=2\ell$ and a polynomial $f \in P^{H_4}_n \backslash
\Id^{H_4}(A)$ alternating in the variables of each set $X_j$.
\end{lemma}
\begin{proof}
By Theorem~\ref{TheoremSweedlerNonSemiSimple}, $A=vJ\oplus J$ (direct sum of subspaces)
where $J^2=0$ and $vJ$ is a $\mathbb Z_2$-simple subalgebra.

Fix the basis $a_1, \ldots, a_\ell, va_1, \ldots, va_\ell$
in $A$ where $a_1, \ldots, a_\ell$ is a basis in $J$.

Since $vJ$ is a $\mathbb Z_2$-simple subalgebra, by~\cite[Theorem~7]{ASGordienko3}, there exist $T \in \mathbb Z_+$ and  $\bar z_1, \ldots, \bar z_T \in vJ$ such that
for any $k \in \mathbb N$
there exists $$f_0=f_0(x_1^1, \ldots, x_\ell^1; \ldots;
x^{2k}_1, \ldots,  x^{2k}_\ell;\ z_1, \ldots, z_T;\ z) \in Q^{F\mathbb Z_2}_{\ell, 2k, 2k\ell+T+1}$$
such that for any $\bar z \in vJ$ we have
$f_0(va_1, \ldots, va_\ell; \ldots;
va_1, \ldots, va_\ell; \bar z_1, \ldots, \bar z_T; \bar z) = \bar z$.

Take $n_0=T+1$, $k=\left[\frac{n-n_0}{2\ell}\right]$, and consider \begin{equation*}\begin{split}f(x_1^1, \ldots, x_{2\ell}^1; \ldots;
x^{k}_1, \ldots,  x^{k}_{2\ell};\ z_1, \ldots, z_T; \ z;\ y_1, \ldots, y_{n-2k\ell-T-1})=\\
\Alt_1 \Alt_2 \ldots \Alt_k
f_0(x_1^1, \ldots, x_\ell^1; \left(x_{\ell+1}^1\right)^v, \ldots, \left(x_{2\ell}^1\right)^v;
 \ldots; \\
x^{k}_1, \ldots,  x^{k}_\ell; \left(x_{\ell+1}^k\right)^v, \ldots, \left(x_{2\ell}^k\right)^v;\ z_1, \ldots, z_T;\ z)\ y_1 y_2 \ldots y_{n-2k\ell-T-1}\in P^{H_4}_n\end{split}\end{equation*}
where $\Alt_i$ is the operator of alternation on the set $X_i:=\lbrace 
x_1^i, \ldots, x_{2\ell}^i\rbrace$.

Now we notice that $$f(va_1, \ldots, va_{\ell}, a_1, \ldots, a_{\ell}; \ldots;
va_1, \ldots, va_{\ell}, a_1, \ldots, a_{\ell};\ \bar z_1, \ldots, \bar z_T; \ 1_A, \ldots, 1_A)
=(\ell!)^{2k} 1_A$$ since $v^2=0$.
The lemma is proved.
\end{proof}

\begin{proof}[Proof of Theorem~\ref{TheoremSweedlerSimpleAmitsur}.]
If $A$ is semisimple, then Theorem~\ref{TheoremSweedlerSimpleAmitsur}
follows from~\cite[Theorem~5]{ASGordienko3}. If $A$ is not semisimple, we 
repeat verbatim the proof of~\cite[Lemma~11 and Theorem~5]{ASGordienko3}
using Lemma~\ref{LemmaSweedlerNSSAltHPolynomial} instead of~\cite[Lemma~10]{ASGordienko3}
and~\cite[Lemma~4]{ASGordienko3} instead of~\cite[Theorem~6]{ASGordienko3}.
\end{proof}

\section*{Acknowledgements}

I am grateful to E.~Jespers, V.\,K.~Kharchenko, and M.\,V.~Zaicev for helpful discussions.


\begin{thebibliography}{99}

\bibitem{BahtGiaZai} Bahturin, Yu.\,A., Giambruno, A., Zaicev, M.\,V.
  $G$-identities on associative algebras.
 {\itshape Proc. Amer. Math. Soc.}, \textbf{127}:1 (1999), 63--69.


\bibitem{BahtZaiGradedExp} Bahturin, Yu.\,A., Zaicev, M.\,V.
Identities of graded algebras and codimension growth.
{\itshape Trans. Amer. Math. Soc.}
\textbf{356}:10 (2004), 3939--3950.


\bibitem{BahtKochMont}
Bahturin, Yu.\,A., Kochetov, M.\,V., Montgomery S.
Group gradings on simple Lie algebras in positive characteristic.
\textit{Proc. Amer. Math. Soc.}, \textbf{137} (2009), 1245--1254 


\bibitem{BahturinLinchenko} Bahturin, Yu.\,A., Linchenko, V.
Identities of algebras with actions of Hopf algebras. \textit{J. Algebra} \textbf{202}:2 (1998), 634--654.

\bibitem{BahturinZaicevSeghal}
Bahturin, Yu.\,A., Zaicev, M.\,V., Sehgal, S.\,K.
Finite-dimensional simple graded algebras.
\textit{Sbornik: Mathematics}, 2008, \textbf{199}:7, 965--983.


\bibitem{BereleHopf} Berele, A. Cocharacter sequences for algebras
with Hopf algebra actions. {\itshape J. Algebra}, \textbf{185}
(1996), 869--885.

\bibitem{Danara} D\u asc\u alescu, S., N\u ast\u asescu, C., Raianu, \c S.
Hopf algebras: an introduction. New York, Marcel Dekker, Inc., 2001.

\bibitem{DoiTakeuchi}
Doi, Y., Takeuchi, M. Quaternion algebras and Hopf crossed products.
\textit{Comm. in Algebra}, \textbf{23}:9 (1995), 3291--3326.


\bibitem{ZaiGia} Giambruno, A., Zaicev, M.\,V.
Polynomial identities and asymptotic methods.
AMS Mathematical Surveys and Monographs Vol. 122,
 Providence, R.I., 2005.
 
     
\bibitem{ASGordienko3} Gordienko, A.\,S. Amitsur's conjecture for associative algebras
with a generalized Hopf action.
\textit{J. Pure and Appl. Alg.},
\textbf{217}:8 (2013), 1395--1411.
         
\bibitem{ASGordienko8} Gordienko, A.\,S. Asymptotics of $H$-identities for associative algebras with an $H$-invariant radical. \textit{J. Algebra}, \textbf{393} (2013), 92--101.
\bibitem{ASGordienko9} Gordienko, A.\,S. Co-stability of radicals and its applications to PI-theory.
\textit{Algebra Colloqium} (to appear).

 
\bibitem{Kharchenko} Kharchenko, V.\,K. Differential identities of semiprime rings.
\textit{Algebra and Logic}, \textbf{18} (1979), 86--119.

 
 \bibitem{MalcevAIFoundations} Malcev, A.\,I. Foundations of linear algebra, San Francisco, W.H. Freeman, 1963.

 \bibitem{Montgomery} Montgomery, S. Hopf algebras and their actions on rings, CBMS Lecture Notes \textbf{82}, Amer. Math. Soc., Providence, RI, 1993.


\bibitem{Sweedler} Sweedler, M.  Hopf algebras. W.A. Benjamin, inc., New York, 1969.

\end{thebibliography}
\end{document}